\documentclass[10pt, twoside, reqno]{amsart}
\usepackage{xcolor, amstext, amsfonts, amssymb, amsbsy, latexsym}
\usepackage{enumerate}
\usepackage[T1]{fontenc}
\usepackage{xy, hhline}
\newcounter{num}[section] %

\newenvironment{theo}{\refstepcounter{num}%
\bigskip\noindent{\bf Theorem~\arabic{section}.\arabic{num}. }\it}

\newenvironment{prop}{\refstepcounter{num}%
\bigskip\noindent{\bf Proposition~\arabic{section}.\arabic{num}. }\it}

\newenvironment{conv}{\refstepcounter{num}%
\bigskip\noindent{\bf Convention~\arabic{section}.\arabic{num}. }\it}

\newenvironment{cor}{\refstepcounter{num}%
\bigskip\noindent{\bf Corollary~\arabic{section}.\arabic{num}. }\it}

\newenvironment{lemma}{\refstepcounter{num}%
\bigskip\noindent{\bf Lemma~\arabic{section}.\arabic{num}. }\it}

\newenvironment{example}{\refstepcounter{num}%
\bigskip\noindent{\bf Example~\arabic{section}.\arabic{num}.}}

\newenvironment{remark}{\refstepcounter{num}%
\bigskip\noindent{\bf Remark~\arabic{section}.\arabic{num}.}}

\newenvironment{conj}{\refstepcounter{num}%
\bigskip\noindent{\bf Conjecture~\arabic{section}.\arabic{num}.}}

\newenvironment{defin}{\refstepcounter{num}%
\bigskip\noindent{\bf Definition~\arabic{section}.\arabic{num}.}}

\newenvironment{proof_of}[1]{\medskip\noindent{\bf Proof #1}}
{$\Box$ \bigskip}
\newenvironment{eq}{\begin{equation}}{\end{equation}}
\renewcommand{\Ref}[1]{(\ref{#1})}

\newcommand{\Char}{\mathop{\rm char}}
\newcommand{\id}[1]{{{\rm id}\{{#1}\}}}

\newcommand{\FF}{\mathbb{F}}
\newcommand{\KK}{\mathbb{K}}

\newcommand{\ZZ}{\mathbb{Z}}

\newcommand{\NN}{\mathbb{N}}
\newcommand{\tq}{ \ | \ }
\newcommand{\FX}{\FF\langle X\rangle}
\newcommand{\Fxy}{\FF\langle x,y\rangle}
\newcommand{\Bxy}{\B\langle x,y\rangle}
\newcommand{\A}{\mathsf{A}}
\newcommand{\B}{\mathsf{B}}
\newcommand{\C}{\mathsf{C}}
\newcommand{\G}{\mathsf{G}}
\newcommand{\R}{\mathsf{R}}
\newcommand{\AW}{\mathcal{A}}
\newcommand{\algA}{\mathcal{A}}
\newcommand{\algB}{\mathcal{B}}
\newcommand{\algC}{\mathcal{C}}
\newcommand{\La}{\mathcal{L}}

\newcommand{\LA}{\langle}
\newcommand{\RA}{\rangle}
\newcommand{\eqPI}{\sim_{\rm PI}}

\newcommand{\M}{\widetilde{M}}

\newcommand{\y}{\widehat{y}}

\newcommand{\D}[2]{\xi_{#1,#2}}     

\newcommand{\Diff}{\mathfrak{D}}
\newcommand{\pa}{\partial}

\newcommand{\si}{\sigma}
\newcommand{\al}{\alpha}
\newcommand{\be}{\beta}
\newcommand{\ga}{\gamma}

\newcommand{\de}{\delta}

\newcommand{\Id}[1]{{{\rm Id}({#1})}}
\newcommand{\IdK}[1]{{{\rm Id}_{\KK}({#1})}}
\newcommand{\ov}[1]{\overline{#1}}
\newcommand{\un}[1]{{\underline{#1}} }
\newcommand{\alg}{\mathop{\rm alg}}
\newcommand{\mdeg}{\mathop{\rm mdeg}}

\usepackage[pdfpagemode=None,
           	pdftitle={[Polynomial Identities for a parametric family of subalgebras of the Weyl algebra},
           	pdfauthor={Artem Lopatin, Carlos Arturo Rodriguez Palma}]{hyperref}

\begin{document}
\title[Identities for a parametric Weyl algebra]{Identities for a parametric Weyl algebra over a ring}


\thanks{The first author was supported by CNPq 313358/2017-6, FAEPEX 2054/19 and FAEPEX 2655/19}

\author{Artem Lopatin}
\address{Artem Lopatin\\ 
State University of Campinas, 651 Sergio Buarque de Holanda, 13083-859 Campinas, SP, Brazil}
\email{dr.artem.lopatin@gmail.com (Artem Lopatin)}

\author{Carlos Arturo Rodriguez Palma}
\address{Carlos Arturo Rodriguez Palma\\ 
State University of Campinas, 651 Sergio Buarque de Holanda, 13083-859 Campinas, SP, Brazil; Industrial University of Santander, Cl.~9 \#Cra 27, Ciudad Universit\'aria, 
Bucaramanga, Santander, Colombia}
\email{carpal1878@gmail.com (Carlos Arturo Rodriguez Palma)}

\begin{abstract} 
In 2013 Benkart, Lopes and Ondrus introduced and studied in a series of papers the infinite-dimensional unital associative algebra $\A_h$ generated by elements $x,y,$ which satisfy the relation $yx-xy=h$ for some $0\neq h\in \FF[x]$. We generalize this construction to $\A_h(\B)$ by working over the fixed  $\FF$-algebra $\B$ instead of $\FF$. We describe the polynomial identities for $\A_h(\B)$ over the  infinite field $\FF$ in case $h\in\B[x]$ satisfies certain restrictions.

\noindent{\bf Keywords: } polynomial identities, matrix identities, Weyl algebra, Ore extensions, positive characteristic.

\noindent{\bf 2020 MSC: } 16R10; 16S32.
\end{abstract}

\maketitle

\section{Introduction}

Assume that $\FF$ is a field of arbitrary characteristic $p=\Char\FF\geq0$. All vector spaces and algebras are over $\FF$  and all algebras are associative, unless other\-wise is stated. For the fixed $\FF$-algebra $\B$ with unity we write $\B\LA x_1,\ldots,x_m\RA$ for the $\FF$-algebra of non-commutative $\B$-polynomials in variables $x_1,\ldots,x_m$, i.e., $\B\LA x_1,\ldots,x_m\RA$ is a free left (and a free right) $\B$-module with the basis given by the set of all non-commutative monomials in $x_1,\ldots,x_m$, where we assume that $\be x_i = x_i \be$ for all $\be\in \B$ and $1\leq i\leq m$. The unity $1$ of $\B\LA x_1,\ldots,x_m\RA$ corresponds to the empty monomial. In case the variables are $x_1,x_2,\dotsc$ the algebra of non-commutative $\B$-polynomials is denoted by  $\B\LA X\RA$. Similarly, we define the algebra of commutative $\B$-polynomials $\B[x_1,\ldots,x_m]$ as a free left (and a free right) $\B$-module with the basis given by the set of all monomials in $x_1^{i_1}\cdots x_m^{i_m}$ with $i_1,\ldots,i_m\geq0$, where we assume that $\be x_i = x_i \be$ and $x_i x_j=x_j x_i$ for all $\be\in \B$ and $1\leq i,j\leq m$. Note that $\B\LA x\RA = \B[x]$.

\subsection{Parametric Weyl algebra $\A_h(\B)$ as the Ore extension}

We study the polynomial identities for the following family of infinite-dimensional unital algebras $\A_h(\B)$, which are parametrized by a polynomial $h$ from the center of $\B[x]$:  

\begin{defin}\label{Def Ah}
For $h\in Z(\B)[x]$, the {\it parametric Weyl algebra} $\A_h(\B)$ {\it over the ring} $\B$ is the unital associative algebra over $\FF$ generated by $\B$ and letters $x$, $y$ commuting with $\B$ subject to the defining  relation $yx=xy+h$ (equivalently, $[y,x]=h$, where $[y,x]=yx-xy$), i.e., $$\A_h(\B)=\Bxy/\id{yx-xy-h}.$$
\end{defin}
\medskip

For short, we denote $\A_h=\A_h(\FF)$. The partial cases of the given construction are the Weyl algebra $\A_1$, the polynomial algebra $\A_0=\FF[x,y]$, and the universal enveloping algebra $\A_x$ of the two-dimensional nonabelian Lie algebra. For $h\in \FF[x]$, the following isomorphism of $\FF$-algebras holds:
\begin{eq}\label{eq_tensor}
\A_h(\B) \simeq \B \otimes_{\FF} \A_h.
\end{eq}%
Note that in general the isomorphism~\Ref{eq_tensor} does not hold because $A_h$ is not well-defined in case $h\not\in\FF[x]$. Given a polynomial  $f= \eta_d x^d + \eta_{d-1} x^{d-1} + \cdots + \eta_0$ of $\B[x]$ with $d\geq0$, we say that  $\eta_d$ is the {\it leading coefficient} of $f$ and the product $\eta_d x^d$ is the {\it leading term} of $f$. 

The algebra $\A_h$  was introduced and studied by Benkart, Lopes, Ondrus~\cite{Benkart_Lopes_Ondrus_I, Benkart_Lopes_Ondrus_II, Benkart_Lopes_Ondrus_III} as a natural object in the theory of Ore extensions. In particular, they determined automorphisms of $\A_h$ over an arbitrary field $\FF$ and the invariants of $\A_h$ under the automorphisms,  completely described the simple modules and derivations of $\A_h$ over any field. Then Lopes and Solotar~\cite{Lopes_Solotar_2019} described the Hochschild cohomology ${\rm HH}^{\bullet}(\A_h)$ over a field of arbitrary characteristic. Over an algebraically closed field of zero characteristic simple $\A_h$-modules were independently classified by Bavula~\cite{Bavula_2020}. In recent preprints~\cite{Bavula_2021_zero},~\cite{Bavula_2021_prime} Bavula continued the study of the automorphism group of $\A_h$.  

Let us recall that an Ore extension of $R$ (or, equivalently, a skew polynomial ring over $R$) $A=\R[y,\sigma,\delta]$ is given by a unital associative (not necessarily commutative) algebra $\R$ over a field $\FF$, an $\FF$-algebra endomorphism $\sigma:\R\rightarrow \R$, and a $\sigma$-derivation $\delta:\R\rightarrow \R$, i.e., $\delta$ is $\FF$-linear map and $\delta(ab)=\delta(a)b+\sigma(a)\delta(b)$ for all $a,b\in \R$. Then $A=\R[y,\sigma,\delta]$ is the unital algebra generated by $y$ over $\R$ subject to the relation $$ya=\sigma(a)y+\delta(a) \quad \text{ for all } a\in \R.$$ 

Assume that $\R=\B[x]$, $\sigma={\rm id}_{\R}$ is the identity automorphism on $\R$, and $\delta:\R\rightarrow \R$ is given by $\delta(a)=a'h$ for all $a\in \R$, where $a'$ stands for the usual derivative of a $\B$-polynomial $a$ with respect to the variable $x$. Since $h\in Z(\R)$, $\de$ is a derivation of $\R$. Using the linearity of derivative and induction on the degree of $a\in\B[x]$ it is easy to see that
\begin{eq}
[y,a]=a'h \text{ holds in }\A_h(\B) \text{ for all }a\in \B[x].
\end{eq}%
\noindent{}Hence $\A_h(\B)=\R[y,\sigma,\delta]$ is an Ore extension. The following lemma is a corollary of Observation 2.1 from~\cite{AVV} proven by Awami, Van den Bergh and Van Oystaeyen (see also Proposition 3.2 of \cite{AD2} and Lemma 2.2 of \cite{Benkart_Lopes_Ondrus_I}).

\begin{lemma} Assume that $A=\R[y,\si,\de]$ is an Ore extension of $\R=\FF[x]$, where $\si$ is an automorphism of $\R$. Then $A$ is isomorphic to one of the following algebras:
\begin{enumerate}
\item[$\bullet$] a quantum plane, i.e., $A\simeq\Fxy/\id{yx-q xy}$ for some $q\in\FF^{*}=\FF\backslash \{0\}$;

\item[$\bullet$] a quantized Weyl algebra, i.e., $A\simeq\Fxy/\id{yx-q xy-1}$ for some $q\in\FF^{*}$;

\item[$\bullet$] an algebra $\A_h$ for some $h\in\FF[x]$.
\end{enumerate}
\end{lemma}
\medskip

\noindent{}Note that by Theorem 9.3 of~\cite{Benkart_Lopes_Ondrus_I} the algebra $\A_h$ is not a generalized Weyl algebra over $\FF[x]$ in the sense of Bavula~\cite{Bavula_1992} in case $h\not\in\FF$.

Since the algebra of $\B$-polynomials $\B[x,y]$ is well studied, in what follows we assume that $h$ is non-zero. Moreover, we assume that the following restriction holds:

\begin{conv}\label{conv1} The leading coefficient of $h\in Z(\B)[x]$ is not a zero divisor.
\end{conv}
\medskip

\subsection{Polynomial identities}

A polynomial identity for a unital $\FF$-algebra $\algA$ is an element $f(x_1,\ldots,x_m)$ of $\FF\LA X\RA$ such that $f(a_1,\ldots,a_m)=0$ in $\algA$ for all $a_1,\ldots,a_m\in \algA$.   The set $\Id{\algA}$ of all polynomial identities for $\algA$ is a T-ideal, i.e.,  $\Id{\algA}$ is an ideal of $\FX$ such that $\phi(\Id{\algA})\subset \Id{\algA}$ for every endomorphism  $\phi$ of $\FX$.  An algebra that satisfies a nontrivial polynomial identity is called a PI-algebra. A T-ideal $I$ of $\FF\LA X\RA$ generated by $f_1,\ldots,f_k\in \FF\LA X\RA$ is the minimal T-ideal of $\FF\LA X\RA$ that contains $f_1,\ldots,f_k$. We say that $f\in \FF\LA X\RA$ follows from $f_1,\ldots,f_k$ if $f\in I$. Given a monomial $w\in \FF\LA x_1,\ldots,x_m\RA$, we write $\deg_{x_i}(w)$ for the number of letters $x_i$ in $w$ and $\mdeg(w)$ for the multidegree $(\deg_{x_1}(w),\ldots,\deg_{x_m}(w))$ of $w$. An element $f\in\FX$ is called multihomogeneous if it is a linear combination of monomials of the same multidegree. We say that algebras $\algA$, $\algB$ are called PI-equivalent and write $\algA \eqPI \algB$ if $\Id{\algA} =\Id{\algB}$.


Denote the $n^{\rm th}$ Weyl algebra by 
$$\AW_n=\FF\LA x_1,\ldots,x_n,y_1,\ldots,y_n\RA/I,$$
where the ideal $I$ is generated by $[y_i,x_j]=\de_{ij}$, $[x_i,x_j]=0$, $[y_i,y_j]=0$ for all $1\leq i,j\leq n$. Note that $\A_1=\AW_1$.

Assume that $p=0$. It is well-known that the algebra $\AW_n$ does not have nontrivial polynomial identities. Nevertheless, some subspaces of $\AW_n$ satisfy certain polynomial identities. Namely, denote by $\AW_n^{(1,1)}$ the $\FF$-span of $x_i y_j$ in $\AW_n$ for all $1\leq i,j \leq n$ and by $\AW_n^{(-,r)}$ the $\FF$-span of $a y_{j_1}\cdots y_{j_r}$ in $\AW_n$ for all $1\leq j_1,\ldots, j_r\leq n$ and $a\in \FF[x_1,\ldots,x_n]$. Dzhumadil'daev~\cite{Askar_2004, Askar_2014} studied the standard polynomial 
$${\rm St}_N(t_{1},\ldots,t_{N})=\sum_{\sigma\in S_{N}}(-1)^{\sigma}t_{\sigma(1)}\cdots t_{\sigma(N)}$$
over some subspaces of $\AW_n$. Namely, he showed that
\begin{enumerate}
\item[$\bullet$] ${\rm St}_N$ is a polynomial identity for $\AW_n^{(-,1)}$ in case $N\geq n^2 + 2n$; 

\item[$\bullet$] ${\rm St}_N$ is not a polynomial identity for $\AW_n^{(-,1)}$ in case $N< n^2 + 2n - 1$;

\item[$\bullet$] ${\rm St}_N$ is a polynomial identity for $\AW_1^{(-,r)}$ if and only if $N>2r$; 

\item[$\bullet$] the minimal degree of nontrivial identity in  $\AW_1^{(-,r)}$ is $2r+1$. 
\end{enumerate}
Using graph--theoretic combinatorial approach Dzhumadil'daev and Yeliussizov~\cite{Askar_Yeliussizov_2015} established that
\begin{enumerate}
\item[$\bullet$] ${\rm St}_{2n}$ is a polynomial identity for $\AW_n^{(1,1)}$ if and only if $n=1,2,3$. 
\end{enumerate}

Note that the space $\AW_n^{(-,1)}$ together the multiplication given by the Lie bracket is the $n^{\rm th}$ Witt algebra $W_n$, which is a simple infinitely dimensional Lie algebra. The polynomial identities for the Lie algebra $W_n$ were studied by Mishchenko~\cite{Mishchenko_1989}, Razmyslov~\cite{Razmyslov_book} and others. The well-known open conjecture claims that all polynomial identities for $W_1$ follow from the standard Lie identity
$$\sum_{\sigma\in S_{4}}(-1)^{\sigma}[[[[t_0,t_{\si(1)}]t_{\si(2)}]t_{\si(3)}]t_{\si(4)}].$$
\noindent{}$\ZZ$-graded identities for $W_1$ were described by Freitas, Koshlukov and  Krasilnikov~\cite{W1_2015}.

\subsection{Results}

In  Theorem~\ref{Teorema Principal} we prove that over an infinite field $\FF$ of positive characteristic $p$ the algebra $\A_h(\B)$ is PI-equivalent to the algebra of $p\times p$ matrices over $\B$ in case $h(\al)$ is not a zero divisor for some $\al\in Z(\B)$. On the other hand, over a finite field the similar result does not hold in case $\B=\FF$ (see Theorem~\ref{theo_finite}). As about the case of zero characteristic, in Theorem~\ref{theo0} we prove that similarly to $\A_1$, the algebra $\A_h(\B)$ does not have nontrivial polynomial identities. 

In Section~\ref{section_example} we consider the algebra $\A_{h}(\B)=\A_{\zeta}(\FF^2)$ such that $h=\zeta$ does not satisfy Convention~\ref{conv1} and the statements of Theorems \ref{theo0} and ~\ref{Teorema Principal} do not hold for $\A_{\zeta}(\FF^2)$. We describe polynomial identities for $\A_{\zeta}(\FF^2)$ and compare them with the polynomial identities for the Grassmann unital algebra of finite rank.

\section{Properties of $\A_h(\B)$}

Many properties of an Ore extension $A=\R[y,\sigma,\delta]$ are inherited from an underlying algebra $\R$. Namely, it is well-known that when $\sigma$ is an automorphism, then:

\begin{itemize} 
    \item $A$ is a free right and a free left $\R$-module with the basis $\{y^i \tq i\geq 0\}$ (see Proposition 2.3 of~\cite{K. R. Goodearl}); 
    
    \item in case $R$ is left (right, respectively) noetherian we have that  $A$ is left (right, respectively) noetherian (see Theorem 2.6 of~\cite{K. R. Goodearl});  
    
    \item in case  $\R$ is a domain we have that $A$ is a domain (see Exercise 2O~of \cite{K. R. Goodearl}).  
\end{itemize}

\noindent{}In case $\B=\FF$ the algebra $\A_h=\A_h(\B)$ is a noetherian domain, but in general case $\A_h(\B)$ lacks these properties, since $\B\subset \A_h(\B)$ (see also Example~\ref{ex2} below).  

In order to distinguish the generators for the algebras $\A_h(\B)$ and $\A_1(\B)$, we will use the following 

\begin{conv}\label{conv2} The generators of $\A_h(\B)$ are denoted by $x,\y, 1$ and the generators of $\A_1(\B)$ are denoted by $x,y,1$.
\end{conv}
\medskip

\begin{lemma}\label{lemma_basis} 
The sets  $\{x^{i}\y^{j}\tq i,j\geq 0\}$ and $\{\y^{j}x^{i}\tq i,j\geq 0\}$ are $\B$-bases for $\A_h(B)$.
\end{lemma}
\begin{proof}
Obviously, $\A_h(\B)$ is the $\B$-span of each of the sets from the lemma. On the other hand, $\B$-linear independence of these sets follows from the fact that $\A_h(\B)$ is a free right and a free left $\B[x]$-module with the basis $\{\y^i \tq i\geq 0\}$.
\end{proof}

Introduce the following lexicographical order on $\ZZ^2$:
$(i,j)<(r,s)$ in case $j<s$ or $j=s,\;i<r$. Denote the multidegree of a monomial $w=x^i\y^j$ of $\A_h(\B)$ by $\mdeg(w)=(i,j)$. Given an arbitrary element   
$a=\sum_{i,j\geq0}\be_{ij} x^i \y^j$ of $\A_h(\B)$, where only finitely many $\be_{ij}\in\B$ are non-zero, define its {\it multidegree} $\mdeg(a)=(d_x,d_y)$ as the maximal multidegree of its momomials, i.e. as the maximal element of the set $\{(i,j) \tq \be_{ij}\neq0\}$. By Lemma~\ref{lemma_basis} the multidegree is well-defined. As above, the coefficient $\be_{d_x,d_y}$ is called the {\it leading coefficient} of $a$ and the product $\be_{d_x,d_y}\,x^{d_x} \y^{d_y}$ is called the {\it leading term} of $a$. In case $a\in\B$ we set $\mdeg(a)=(0,0)$ and the leading coefficient as well as the leading term of $a$ is $a$. 

\begin{lemma}\label{lemma_mult} 
Assume $i,j,r,s\geq 0$. Then 
\begin{enumerate}
\item[(a)] the leading term of $x^i \y^j\cdot x^r \y^s$ is $x^{i+r}\, \y^{j+s}$;

\item[(b)] in case $h\in\B$ we have
$$x^i \y^j\cdot x^r \y^s = \sum_{k=0}^{\min\{j,r\}} k! \binom{j}{k} \binom{r}{k} x^{i+r-k} h^k\, \y^{j+s-k}.$$ 
\end{enumerate}
\end{lemma}
\begin{proof}
Recall that $\de(a)=a'h$ for each $a\in\B[x]$. Since $[\y,a]=\de(a)$, the induction on $j$ implies that 
\begin{eq}\label{eq1_lemma_mult}
\y^j x^r = \sum_{k=0}^{j} \binom{j}{k} \de^k(x^r) \y^{j-k}
\end{eq}%
(cf.~Lemma 5.2 of~\cite{Benkart_Lopes_Ondrus_I}). Taking $k=0$ in equality~\Ref{eq1_lemma_mult}, we obtain that the leading term of $\y^j x^r$ is $x^r\y^j$. Similarly we conclude the proof of part (a). Part (b) follows from~\Ref{eq1_lemma_mult} and 
$$\de^k(x^r) = 
\left\{
\begin{array}{cl}
\frac{r!}{(r-k)!} x^{r-k} h^k, & \text{if } k\leq r \\
0,  &  \text{if } k> r \\
\end{array}
\right..
$$
\end{proof}

\begin{lemma}\label{lemma_deg} 
If the leading coefficient of one of non-zero elements $a,b\in \A_h(\B)$ is not a zero divisor, then  $\mdeg(ab) = \mdeg(a) + \mdeg(b)$. In particular, $ab$ is not zero.
\end{lemma}
\begin{proof} Consider $a=\sum_{i=1}^m \be_{i} x^{r_i} \y^{s_i}$ and $b=\sum_{j=1}^n \ga_{j} x^{k_j} \y^{l_j}$ for $m,n\geq1$ and non-zero $\be_i,\ga_j\in\B$, where we assume that elements of each of the sets $\{(r_i,s_i)\tq 1\leq i\leq m\}$ and  $\{(k_j,l_j)\tq 1\leq j\leq n\}$ are pairwise different. Assume that $\mdeg(a)=(r_1,s_1)$ and $\mdeg(b)=(k_1,l_1)$. Part (a) of  Lemma~\ref{lemma_mult} implies that 
$$\mdeg(x^{r_1} \y^{s_1}  x^{k_1} \y^{l_1})=(r_1\!+\!k_1, s_1\!+\!l_1) \;\text{ and }\; \mdeg(x^{r_i} \y^{s_i}  x^{k_j} \y^{l_j})<(r_1\!+\!k_1, s_1\!+\!l_1)$$ if $(i,j)\neq (1,1)$. Since $\be_1 \ga_1\neq 0$, we obtain $\mdeg(ab) = (r_1+k_1, s_1 + l_1)$ and the proof is concluded. 
\end{proof}

\begin{lemma}\label{lemma_embedding} 
\begin{enumerate}
\item[(a)] The $\B$-linear homomorphism  of $\FF$-algebras  $\psi: \A_h(\B)\to \A_1(\B)$, defined by 
$$1\to 1,\quad x\to x,\quad \y\to y h,$$
is an embedding $\A_h(\B)\subset \A_1(\B)$. 

\item[(b)] $\{x^{i}h^{j}y^{j}\tq i,j\geq 0\}$ and $\{y^{j}h^{j}x^{i}\tq i,j\geq 0\}$ are $\B$-bases for $\A_h(\B)\subset \A_1(\B)$.
\end{enumerate}
\end{lemma}
\begin{proof}
\noindent{\bf (a)} Since $\psi([\y,x]-h)=([y,x]-1)h=0$ in $\A_1(\B)$, the homomorphism $\psi$ is well-defined. Assume that $\psi$ is not an embedding, i.e., there exists non-zero finite sum $a=\sum_{i,j\geq0}\be_{ij} x^i \y^j$ with $\be_{ij}\in \B$ such that
$$
\psi(a)=\sum_{i,j\geq0}\be_{ij} x^i (y h)^j =0 \quad\text{ in }\quad\A_1(\B).
$$

Denote $\mdeg(a)=(r,s)$. If $(r,s)=(0,0)$, then $a\in \B$ and $\psi(a)=a$ is not zero; a contradiction. Therefore, $(r,s)\neq (0,0)$.
Since $\mdeg(x^i (y h)^j)= (i + j \deg(h), j)$ by Lemma~\ref{lemma_deg} and Convention~\ref{conv1}, we obtain that $\mdeg(\psi(a))= (r + s \deg(h), s)$ is not zero; a contradiction.   
\medskip

\noindent{\bf (b)} Since $h$ lies in the center of $\B[x]$, repeating the proof of Lemma 3.4 from~\cite{Benkart_Lopes_Ondrus_I} for $\A_h(\B)$ we can see that
$$\A_h(\B)=\bigoplus_{j\geq 0}\B[x]h^{j}y^{j}=\bigoplus_{j\geq 0}y^{j}h^{j}\B[x].$$
Similarly to part (a), we conclude the proof by the reasoning with multidegree.
\end{proof}

\begin{example}\label{ex2} Assume $\B$ is the $\FF$-algebra of double numbers, i.e., $\B$ has an $\FF$-basis $\{1,\zeta\}$ with $\zeta^2=0$. Then the ideal $I=\FF\text{-span}\{\zeta x^i \y^j \tq i,j\geq 0\}$ is a proper nilpotent ideal of $\A_h(\B)$. In particular, the algebra $\A_h(\B)$ is not semi-prime.
\end{example}
\medskip

\section{$\A_h(\B)$ as the algebra of differential operators}\label{section_diff}

Denote by ${\rm Map}(\B[z])$ the algebra of all $\FF$-linear maps over $\B[z]$ with respect to composition. Assume that  $\Diff_h(\B[z])$ is the subalgebra of ${\rm Map}(\B[z])$ generated by the following maps: the multiplication $\be\,{\rm Id}$ by an element $\be$ of $\B$, i.e., $(\be\,{\rm Id})(f)=\be f$, the multiplication $\chi$ by $z$, i.e., $\chi(f)=zf$, and the derivation $\de$ given by $\de(f)=f'h(z)$ for all $f\in\B[z]$. Note that $\de= h(\chi)\, \pa$, where $\pa$ stands for the operator of the usual derivative. Obviously, maps $\chi$, $h(\chi)$ and $\pa$ are $\B$-linear. For short, we write $\chi^0$ for ${\rm Id}$. 

\begin{prop}\label{prop_diff}
\begin{enumerate}
\item[(a)]  $\{\chi^{i} h(\chi)^j \pa^{j}\tq i,j\geq 0\}$ is an $\B$-basis for $\Diff_h(\B[z])$ in case $p=0$. 

\item[(b)]  $\{\chi^{i}  h(\chi)^j \pa^{j}\tq 0\leq i,\; 0\leq j<p \}$ is an $\B$-basis for $\Diff_h(\B[z])$ in case $p>0$. 

\item[(c)]$\A_h(\B) /\id{h^p y^p} \simeq \Diff_h(\B[z])$ for each $p\geq 0$.
\end{enumerate}
\end{prop}
\begin{proof} 

Consider the $\B$-linear homomorphism of $\FF$-algebras $\Phi: \A_1(\B) \to {\rm Map}(\B[z])$ given by 
$1\to {\rm Id}$, $x\to \chi$, $y\to \pa$. Since $\Phi([y,x]-1)(f)=(\pa \chi - \chi \pa - {\rm Id})(f) = f + z f'  - z f'  -  f = 0$ for all $f\in\B[z]$, the map $\Phi$ is a well-defined.  Applying $\Phi$ to parts (a) and (b) of Lemma~\ref{lemma_embedding} we obtain that   $\Diff_h(\B[z]) = \Phi(\A_h(\B))$ is an $\B$-span of $\{\chi^{i} h(\chi)^j \pa^{j}\tq i,j\geq 0\}$.

Let $p=0$. Assume that some non-zero finite sum $\pi=\sum_{i,j\geq 0} \be_{ij} \chi^{i} h(\chi)^j \pa^{j}$ with $\be_{ij}\in\B$ belongs to the kernel of $\Phi$. Denote by $j_0$ the minimal $j\geq0$ with $\be_{ij}\neq 0$ for some $i$ and denote by $i_0$ the maximal $i$ with $\be_{ij_0}\neq 0$. Then $\pi(z^{j_0}) = j_0!\, h(z)^{j_0} \sum_{0\leq i\leq i_0} \be_{i,j_0}\,  z^i =0$ in $\B[z]$. Thus Convention~\ref{conv1} together with $j_0!\,\be_{i_0,j_0} \neq 0$ implies  a contradiction. Part~(a) is proven.

Assume that $p>0$ and some non-trivial finite sum 
$$\pi=\sum\limits_{0\leq i,\; 0\leq j<p } \be_{ij} \chi^{i} h(\chi)^j \pa^{j}$$ 
with $\be_{ij}\in\B$ belongs to the kernel of $\Phi$. As above, we obtain a contradiction. Since $\pa^p=0$, we conclude the proof of part (b).

Parts (a) and (b) together with  part (b) of Lemma~\ref{lemma_embedding} conclude the proof of part (c).
\end{proof}


\begin{theo}\label{theo0}
In case $p=0$ the algebra $\A_h(\B)$ does not have nontrivial  polynomial identities. 
\end{theo}
\begin{proof}
Assume that $\FF$-algebra $\A_h(\B)$ has a nontrivial polynomial identity. Since $p=0$, there exists $N>0$ such that $\A_h(\B)$ satisfies a nontrivial multilinear identity $f(x_1,\ldots,x_N)=\sum_{\si\in S_N}\al_{\si} x_{\si(1)}\cdots x_{\si(N)}$ with $\al_{\si}\in \FF$. Moreover, we can assume that  $\al_{\rm Id}\neq0$ for the identity permutation ${\rm Id}$. Given $j>0$, we write $F_j$ for a $\B$-linear map  $\chi^{2j} h(\chi)^j \pa^{j}$ from $\Diff_h(\B[z])$.  Denote by $d\geq0$ the degree of $h$ and we write $\eta$ for the leading coefficient of $h$. Recall that $\eta$ is not a zero divisor  by Convention~\ref{conv1}.  Note that
\begin{eq}\label{eq_theo0}
F_{j}(z^m)=\left\{
\begin{array}{cc}
0, & \text{ in case }m<j\\
\frac{m!}{(m-j)!} \, z^{m+j} h(z)^j, & \text{ in case }m\geq j.\\
\end{array}
\right.
\end{eq}%
In particular, $\deg(F_{j}(z^j))=j(d+2)$ and the leading coefficient of $F_{j}(z^j)$ is $j!\,\eta^j$, which is not a zero divisor. 

By parts (a) and (c) of Proposition~\ref{prop_diff}, the equality $f(F_{j_{1}},\ldots,F_{j_N})=0$ holds in $\Diff_h(\B[z])$ for all $j_1,\ldots,j_{N}>0$.  Consider $j_k=(d+2)^{N-k}$ for all $0\leq k\leq N$. 
Note that $1=j_N<j_{N-1}<\cdots < j_0$. We claim that for any $\si\in S_N$ we have
\begin{eq}\label{claim_theo0}
F_{j_{\si(1)}}\circ \cdots \circ F_{j_{\si(N)}}(z) \neq 0 \text{ if and only if } \si= {\rm Id}.
\end{eq}%
\begin{eq}\label{claim2_theo0}
\text{The leading term of } F_{j_1}\circ \cdots \circ F_{j_N}(z) \text{ is } j_{1}!\cdots j_N!\,\eta^{j_{1}+\cdots +j_N} z^{j_{0}}.
\end{eq}%
Let us prove these claims. Assume that $F_{j_{\si(1)}}\circ \cdots \circ F_{j_{\si(N)}}(z) \neq 0$ for some $\si\in S_N$. 

Since $F_{j_{\si(N)}}(z)\neq 0$, then equality~\Ref{eq_theo0} implies that $\si(N)=N$, $j_{\si(N)}=1$,  $\deg(F_{j_{\si(N)}}(z))=d+2=j_{N-1}$ and the leading coefficient of $F_{j_{\si(N)}}(z)$ is $\eta$, which is not a zero divisor. 

Similarly, assume that for $1\leq l< N$ with $\si(l)\leq l$ the inequality $F_{j_{\si(l)}}(g)\neq 0$ holds for some $g\in \B[z]$ with the leading term $j_{l+1}!\cdots j_N!\,\eta^{j_{l+1}+\cdots + j_N} z^{j_l}$. Then equality~\Ref{eq_theo0} implies that $\si(l)=l$ and $\deg(F_{j_{\si(l)}}(g))=j_{l-1}$. Moreover, the leading term of $F_{j_{\si(l)}}(g)$ is $j_{l}!\cdots j_N!\,\eta^{j_{l}+\cdots +j_N} z^{j_{l-1}}$. Consequently applying this reasoning to $l=N-1,\, l=N-2,\ldots, l=1$, we conclude the proof of claims~\Ref{claim_theo0} and~\Ref{claim2_theo0}.

Claims~\Ref{claim_theo0} and~\Ref{claim2_theo0} imply that  $0=f(F_1,\ldots,F_N)(z)=\al_{\rm Id}\, F_{j_1}\circ \cdots \circ F_{j_N}(z)\neq 0$ by Convention~\ref{conv1}; a contradiction.
\end{proof}

\section{Polynomial Identities for $\A_h(\B)$ in positive characteristic}\label{section_positive}
 
We write  $M_n=M_n(\FF)$ for the algebra of all $n\times n$ matrices over $\FF$ and denote by $\M_n$ the algebra of all $n\times n$ matrices over $\B[x,y]$. Denote by $I_n$ the identity $n\times n$ matrix and by $E_{ij}\in M_n$ the matrix such that the $(i,j)^{\rm th}$ entry is equal to one and the rest of entries are zeros. Consider the properties of the next two matrices of $M_p$:  
$$A_{0}=\sum_{i=1}^{p-1}E_{i+1,i}\;\;\text{ and }\;\;B_{0}=\sum_{i=1}^{p-1}i\cdot E_{i,i+1}.$$

\begin{lemma}\label{lemma_A0_B0}
\begin{enumerate}
\item[(a)]  For all $0\leq k<p$ we have that
$$A_{0}^{k}=\sum_{i=1}^{p-k}E_{k+i,i}\;\;\text{ and }\;\;  B_{0}^{k}=\sum_{i=1}^{p-k}\frac{(k+i-1)!}{(i-1)!}E_{i,k+i},$$

\noindent{}where $A_0^0$ and $B_0^0$ stand for $I_p$.


\item[(b)] $B_{0}A_{0}-A_{0}B_{0}=I_{p}$. 
\end{enumerate}
\end{lemma}	
\begin{proof} The formula for $A_{0}^{k}$ is trivial. We prove the formula for $B_{0}^{k}$ by induction on $k$. For $k=1$ the claim holds. Assume that the claim is valid for some $k<p-1$. Then
	\begin{align*}
	B_{0}^{k+1}&=\left(\sum_{i=1}^{p-k}\frac{(k+i-1)!}{(i-1)!}E_{i,k+i}\right)\left(\sum_{r=1}^{p-1}r\cdot E_{r,r+1} \right)=\sum_{i=1}^{p-(k+1)}\frac{(i+k)!}{(i-1)!}E_{i,k+1+i}
	\end{align*}
and the required is proven. Part (b) is straightforward. 
\end{proof}
	
Define the $\B$-linear homomorphism $\varphi:\B\LA x,y \RA \rightarrow \M_p$ of algebras by 
$$x\mapsto A,\;\;  y\mapsto B,\;\; 1\mapsto I_p,$$ 
where $A=xI_{p}+A_{0}$ and $B=yI_{p}+B_{0}$. Since $\be A=A \be$ and $\be B = B\be$ for each $\be\in B$, the homomorphism $\varphi$ is well-defined.
 	
\begin{lemma}\label{A1 subset Map}
The  homomorphism $\varphi$  induces the injective $\B$-linear homomorphism $\ov{\varphi}:\A_1(\B)\rightarrow \M_p$. In particular,  the restriction of $\ov{\varphi}$ to $\A_h(\B)\subset \A_1(\B)$ is the injective $\B$-linear homomorphism $\A_h(\B)\rightarrow \M_p$.
\end{lemma}
		
\begin{proof}
By part (b) of Lemma \ref{lemma_A0_B0} we have that $\varphi(yx-xy-1)=BA-AB-I_{p}=0$. Therefore, $\varphi$ induces a $\B$-linear homomorphism $\ov{\varphi}:\A_1(\B)\rightarrow \M_p$ of algebras. 

Assume that there exists a nonzero $a\in \A_1(\B)$ such that $\ov{\varphi}(a)=0$. Since  $\{x^{i}y^{j}\tq i,j\geq 0\}$  is an $\B$-basis for $\A_1(\B)$ by Lemma~\ref{lemma_basis}, we have $a=\sum_{i,j\geq 0}\be_{ij}x^{i}y^{j}$ for a finite sum with $\be_{ij}\in \B$. Thus 
$0= \ov{\varphi}(a) = \sum_{i,j\geq 0} \be_{ij} A^i B^j$. The equalities 
$$A^{i}=\begin{pmatrix}
	x^{i} & 0 &\cdots & 0 \\
	* & * &\cdots & * \\
	\vdots &\vdots&&\vdots\\
	* & * &\cdots & * \\
	\end{pmatrix}
	\;\; \text{and} \;\;
	B^{j}=\begin{pmatrix}
	y^{j} & * &\cdots & * \\
	0 & * &\cdots & * \\
	\vdots &\vdots&&\vdots\\
	0 & * &\cdots & * \\
	\end{pmatrix}$$ 
imply that the $(1,1)^{\rm th}$ entry of $A^{i}B^{j}$ is $(A^{i}B^{j})_{1,1}=x^{i}y^{j}$. Hence
$0 = (\varphi(a))_{1,1} = \sum_{i,j\geq 0} \be_{ij} x^{i}y^{j}$ in $\B[x,y]$. 
Hence $\be_{ij}=0$ for all $i,j\geq 0$ and $a=0$; a contradiction. Therefore $\ov{\varphi}$ is injective.
\end{proof}

Given $1\leq i,j\leq p$ and $k\geq 1$, we write $z_{ij}(k)$ for $x_{i+p(j-1)+p^2(k-1)}\in \FF\LA X\RA$. The {\it generic $p\times p$ matrix $X_k$ with non-commutative elements} is the matrix  $X_k=(z_{ij}(k))_{1\leq i,j\leq p}$.
				
\begin{cor}\label{Id(Mp) subset Id(Ah)} 
$\Id{M_p(\B)}\subset \Id{\A_h(\B)}$, if the field $\FF$ is infinite. 
\end{cor}
\begin{proof}
Lemma~\ref{A1 subset Map} implies that $\Id{\M_p}\subset \Id{\A_h(\B)}$.

Since $B\subset B[x,y]$, we have $\Id{\M_p}\subset \Id{M_p(\B)}$. On the other hand, assume that $f=f(x_1,\ldots,x_n)$ is a polynomial identity for $M_p(\B)$. Then $f(X_1,\ldots,X_n)= (f_{ij})_{1\leq i,j\leq n}$ for some $f_{ij}\in \FF\LA X\RA$ with $f_{ij}\in \Id{\B}$.  It is well-known that for an infinite field $\FF$ and a commutative unital $\FF$-algebra $\algC$ the polynomial identities for a unital $\FF$-algebra $\algB$ and $\algC\otimes_{\FF} \algB$ are the same (for example, see Lemma 1.4.2 of~\cite{A.Gia}). Since $\B[x,y]=\FF[x,y]\otimes_{\FF} \B$, we obtain $f_{ij}\in \Id{\B[x,y]}$ and $f\in \Id{\M_p}$. The required is proven.
\end{proof}
		

For each $\al\in Z(\B)$ consider the evaluation $\B$-linear homomorphism  $\epsilon_{\al}:\B[x,y]\rightarrow \B$ of unital $\FF$-algebras defined by 
$$x\mapsto \al, \;\; \ y\mapsto 0$$
and extend it to the evaluation homomorphism $\varepsilon_{\al}:\M_p\rightarrow M_p(\B)$. Since $\A_h(\B)$ is a subalgebra of $\M_p$ by means of embedding $\ov{\varphi}$ (see Lemma \ref{A1 subset Map}), we can consider the images of $x,\y\in \A_h(\B)$  in $M_p(\B)$, which we denote by $C_{\al}$ and $D_{\al}$, respectively: 
$$\begin{array}{cll}
C_{\al}=\varepsilon_{\al}(\ov{\varphi}(x))& =\varepsilon_{\al}(A) & = \al I_p + A_0,\\
D_{\al}=\varepsilon_{\al}(\ov{\varphi}(\y))& =\varepsilon_{\al}(Bh(A)) & = B_{0}\varepsilon_{\al}(h(A)).
\end{array}
$$

\noindent{}Obviously,  $\be C_{\al} = C_{\al} \be$ and $\be D_{\al} = D_{\al} \be$ for each $\be\in B$. To obtain the explicit description of the matrix $D_{\al}$ we calculate $h(A)$. For  $r\geq1$ denote the  $r^{\rm th}$ derivative of $h\in\B[x]$ by $h^{(r)}=\frac{d^{r}h}{dx^{r}}$ and write $h^{(0)}$ for $h$. Note that $h^{(r)}(\al)$ lies in the center of $\B$.

\begin{lemma}\label{lemma_hA}
$$h(A) =\sum_{i=1}^{p}\sum_{j=1}^{i} \frac{1}{(i-j)!} h^{(i-j)} E_{ij}.$$
\end{lemma}
		
\begin{proof} 
We start with the case of $h=x^{k}\in\B[x]$ for some $k\geq0$. Obviously, the claim of the lemma holds for $h=1$. Therefore, we assume that $k\geq1$. Since $A_{0}^{r}=0$ for all $r\geq p$, we have 
$$h(A) = A^{k} = (xI_{p}+A_{0})^{k} = \sum_{r=0}^{\min\{k,p-1\}}\binom{k}{r}x^{k-r}A_{0}^{r}.$$

\noindent{}Part (a) of Lemma \ref{lemma_A0_B0} implies
	
$$h(A) = \sum_{r=0}^{\min\{k,p-1\}}\binom{k}{r}x^{k-r} \left(\sum_{i=1}^{p-r}E_{r+i,i}\right).$$

\noindent{}Regrouping the terms we obtain

\begin{eq}\label{eq_lemma_hA}
h(A)=\sum_{i=1}^{p}\; \sum_{j=\max\{1,i-k\}}^{i} \binom{k}{i-j}x^{k-(i-j)}E_{ij}.
\end{eq}
 
\noindent{}Note that for $0\leq r< p$ we can rewrite 
$$\frac{1}{r!} h^{(r)} = \left\{
\begin{array}{rl}
\binom{k}{r}x^{k-r}, & \text{ if } r\leq k \\
0, & \text{ otherwise } \\
\end{array}
\right.,
$$
where $r!$ is not zero in $\FF$.  Hence equality~\Ref{eq_lemma_hA} implies that the claim holds for $h=x^{k}$.

The general case follows from the proven partial case and the $\B$-linearity of derivatives.
\end{proof}

Lemma~\ref{lemma_hA} together with the definition of $D_{\al}$ immediately implies the next corollary.

\begin{cor}\label{cor_D0}
$$D_{\al} = \sum_{i=1}^{p-1} \sum_{j=1}^{i+1} \frac{i}{(i-j+1)!} h^{(i-j+1)}(\al) E_{ij}.$$
\end{cor}
\bigskip

For short, denote the $(i,j)^{\rm th}$ entry of $D_{\al}$ by $\D{i}{j}\in Z(\B)$ and for all $1\leq k < p$ define 
\begin{eq}\label{eq_Dak}
D_{\al,k} = D_{\al} - \sum_{r=0}^{k-1}\D{k}{k-r}A_0^{r} = D_{\al} - \sum_{r=0}^{k-1} \sum_{i=1}^{p-r} \D{k}{k-r}E_{r+i,i}.
\end{eq}%
\noindent{}We apply the following technical lemma in the proof of key Proposition~\ref{Id(Ah) subset Id(Mp)} (see below). 

\begin{lemma}\label{Epk.Dk}
For all $1\leq r\leq p$ and $1\leq k < p$ we have 
$$E_{rk}\, D_{\al,k}=k\, h(\al) E_{r,k+1}.$$
\end{lemma}
		
\begin{proof}
We have
\begin{align*}
	E_{rk}D_{\al,k}&=\sum_{i=1}^{p-1} \sum_{j=1}^{i+1} \D{i}{j}E_{rk}E_{ij} - \sum_{j=0}^{k-1}\sum_{i=1}^{p-j} \D{k}{k-j} E_{rk} E_{i+j,i}\\
	&= \sum_{j=1}^{k+1} \D{k}{j}E_{rj} - \sum_{j=1}^{k} \D{k}{j} E_{rj}\\
	&= \D{k}{k+1}E_{r,k+1}.\\
	\end{align*} 
Equality $\D{k}{k+1} = k\, h(\al)$ concludes the proof.
\end{proof}

\begin{prop}\label{Id(Ah) subset Id(Mp)} 
\begin{enumerate}
\item[(a)] Assume $\al\in Z(\B)$. Then $\varepsilon_{\al}(\A_h(\B))$ contains $h(\al)^{2(p-1)} M_p(\B)$.

\item[(b)] Assume $h(\al)$ is invertible in $\B$ for some $\al\in Z(\B)$. Then $\Id{\A_h(\B)}\subset \Id{M_p(\B)}$.

\item[(c)] Assume $h(\al)$ is not a zero divisor for some $\al\in Z(\B)$ and $\FF$ is infinite. Then $\Id{\A_h(\B)}\subset \Id{M_p(\B)}$.	
\end{enumerate}
\end{prop}
		
\begin{proof} For short, we write $\be$ for $h(\al)\in Z(\B)$.

\medskip
\noindent{\bf (a)} 
Denote by $\La = \varepsilon_{\al}(\A_h(\B)) = \alg_{\B}\{I_p,C_{\al}, D_{\al}\}$ the $\FF$-algebra generated by $\B I_p$, $C_{\al}$, $D_{\al}$. Since $A_0=C_{\al} - \al I_p$, we obtain that 
$$A_0^{k}=\sum_{i=1}^{p-k}E_{k+i,i}\in \La$$
for all $0\leq k<p$.  In particular, $E_{p1}=A_0^{p-1}\in\La$. Equality~\Ref{eq_Dak} implies that $D_{\al,k}\in\La$ for all $1\leq k < p$.

The statement of part (a) follows from the following claim:
\begin{eq}\label{claim1}
\{ \be^{p+k-r-1} E_{rk} \tq 1\leq r,k \leq p\}\subset \La. 
\end{eq}%
\noindent{}To prove the claim we use descending induction on $r$. 

Assume $r=p$. We have $E_{p1}\in\La$. Lemma~\ref{Epk.Dk} implies that $E_{p1} D_{\al,1} = \be E_{p2}$.  Since $E_{p1}$, $D_{\al,1}$ belong to $\La$,  we can see that $\be E_{p2}\in\La$. Similarly, the equality $\be E_{p2}D_{\al,2} = 2 \be^2 E_{p3}$ implies $\be^2  E_{p3}\in\La$. Repeating this reasoning we obtain that $\be^{k-1} E_{pk}\in\La$ for all $1\leq k\leq p$.

Assume that for some $1\leq r< p$ claim~\Ref{claim1} holds for all $r'>r$, i.e.,  for every $1\leq k\leq p$ we have $\be^{p+k-r'-1}  E_{r'k}\in\La$.  Since
$$\be^{p-r}\left(A_0^{r-1} - \sum_{k=2}^{p-r+1} E_{(r-1)+k,k}\right) = \be^{p-r} E_{r1},$$
we obtain $\be^{p-r} E_{r1}\in\La$. Lemma~\ref{Epk.Dk} implies that $\be^{p-r} E_{r1} D_{\al,1} = \be^{p+1-r} E_{r2}$.  Hence $\be^{p+1-r} E_{r2}\in\La$.  Repeating this reasoning we obtain that $\be^{p+k-r-1} E_{rk} \in\La$ for all $1< k\leq p$, since $\be^{p+k-r-2} E_{r,k-1} D_{\al,k-1} = (k-1)\, \be^{p+k-r-1} E_{r,k}$. Claim~\Ref{claim1} is proven. 

\medskip
\noindent{\bf (b)} Since $h(\al)$ is invertible in $\B$, part (a) implies that $\varepsilon_{\al}(\A_h(\B))=M_p(\B)$. Since $\varepsilon_{\al}$ is a homomorphism of $\FF$-algebras, the required is proven. 

\medskip
\noindent{\bf (c)} Consider a polynomial identity $f\in \FF\LA x_1,\ldots,x_m\RA$ for $\A_h(\B)$. Since $\FF$ is infinite, without loss of generality we can assume that $f$ is homogeneous with respect to the natural grading of $\FF\LA x_1,\ldots,x_m\RA$ by degrees, i.e., each monomial of $f$ has one and the same degree $t>0$. 
Part (a) implies that for every $A_1,\ldots,A_m$ from $M_p(\B)$ there exist $a_1,\ldots,a_m$ from $A_h(\B)$ such that
$$\be^{2(p-1)t} f(A_1,\ldots,A_m) = f(\be^{2(p-1)} A_1,\ldots,\be^{2(p-1)} A_m) = f(\varepsilon_{\al}(a_1),\ldots,\varepsilon_{\al}(a_m)).$$
Since $\varepsilon_{\al}$ is a homomorphism of $\FF$-algebras, we have $f(\varepsilon_{\al}(a_1),\ldots,\varepsilon_{\al}(a_m))=0$. Therefore  $f$ is a polynomial identity for $\A_h(\B)$ because $\be$ is not a zero divisor.

\end{proof}

To illustrate the proof of part (a) of Proposition~\ref{Id(Ah) subset Id(Mp)}, we repeat it in the partial case of $p=3$ in the following example. 

\begin{example}\label{ex1}
Assume $p = 3$ and $h(\al)\neq0$ for some $\al\in Z(\B)$. For short, denote $\be=h(\al)$, $\be'=h'(\al)$ and $\be''=h''(\al)$. Then $A_0=E_{21} + E_{32}$, 
$$ 		C_{\al}=\begin{pmatrix}
	\al& 0      & 0\\
	1     & \al & 0\\
	0     & 1      & \al \\
\end{pmatrix},\quad\text{ and }\quad 
D_{\al}=\begin{pmatrix}
	\be'&\be&0\\
	\be''& 2\be' & 2\be\\
	0&0&0
\end{pmatrix}.
	$$%
To show that $\La = \alg_{\B}\{I_p,C_{\al}, D_{\al}\}$ contains  $\be^4 M_3(\B)$, we consider the following elements of $\La$: 
$$  D_{\al,1}=D_{\al} - \be' I_3= 
\begin{pmatrix}
	0        & \be  & 0\\
	\be'' & \be' & 2 \be\\
	0        & 0       & -\be' \\
	\end{pmatrix},
	$$
$$  D_{\al,2}=D_{\al} - 2 \be' I_3 - \be'' A_0 = 
\begin{pmatrix}
	-\be' & \be  & 0\\
	0        & 0       & 2 \be\\
	0        & -\be'' & -2 \be' \\
	\end{pmatrix}.
	$$%
\noindent{}Note that $A_0=C_{\al} - \al I_3$ and $A_0^2=E_{31}$ belong to $\La$. Since 
$$E_{31} D_{\al,1} = \be E_{32}\;\;\text{ and }\;\; \be E_{32} D_{\al,2} = 2 \be^2 E_{33},$$ 
we obtain that $ \be E_{32},\, \be^2E_{33}\in\La$. Thus $ \be(A_0-E_{32})= \be E_{21}$ lies in $\La$. Since 
$$\be E_{21} D_{\al,1} = \be^2 E_{22}\;\;\text{ and }\;\; \be^2 E_{22} D_{\al,2} = 2 \be^3 E_{23},$$ 
we obtain 
that $ \be^2 E_{22},\,  \be^3 E_{23}\in\La$. Hence $ \be^2(I_3 - E_{22} - E_{33}) = \be^2 E_{11}$ lies in $\La$. Since 
$$\be^2 E_{11} D_{\al,1} = \be^3 E_{12}\;\;\text{ and }\;\; \be^3 E_{12} D_{\al,2} = 2 \be^4 E_{13},$$ we obtain 
that $\be^3 E_{12},\, \be^4 E_{13}\in\La$. Therefore, $\La$ contains $\be^4 M_3(\B)$.
\end{example}
\medskip


\begin{theo}\label{Teorema Principal}
Assume that $\FF$ is an infinite field of characteristic $p>0$.
\begin{enumerate}
\item[(a)] If $h(\al)$ is not a zero divisor for some $\al\in Z(\B)$, then $\A_h(\B)\eqPI M_p(\B)$.

\item[(b)] If $\B=\FF$, then $\A_h\eqPI M_p$.
\end{enumerate}
\end{theo}
\begin{proof} Part (a) follows from Corollary \ref{Id(Mp) subset Id(Ah)} and part (c) of Proposition \ref{Id(Ah) subset Id(Mp)}. 

Assume that $\B=\FF$. Then there exists $\al\in\FF$ with $h(\al)\neq0$, because $h\in\FF[x]$ is not zero and $\FF$ is infinite. Part (a) concludes the proof of part (b).
\end{proof}
\medskip

\begin{cor}\label{cor1}
Assume that $\FF$ is an infinite field of characteristic $p>0$ and for $h= \eta_d x^d + \eta_{d-1} x^{d-1} + \cdots + \eta_0$ from $Z(\B)[x]$ we have that $\eta_d$ and $\eta_0$ are not zero divisors. Then  $\A_h(\B)\eqPI M_p(\B)$.
\end{cor}
\bigskip

\section{$\A_h$ over finite fields}\label{section_finite}

In this section we assume that $\B=\FF$ is the field of finite of order $q=p^k$ and $\FF\subset \KK$ for an infinite field $\KK$. Since $h\in\FF[x]$, Convention~\ref{conv1} is equivalent to the inequality $h\neq0$. As above, we write $M_p$ for $M_p(\FF)$ and $\A_h$ for $\A_h(\FF)$.   Given a $\KK$-algebra $\algA$, we write $\IdK{\algA}$ for the ideal of $\KK\LA X\RA$ of polynomial identities for $\algA$ over $\KK$. In this section we proof the next result.

\begin{theo}\label{theo_finite} 
\begin{enumerate}
\item[(a)]  $\IdK{M_p(\KK)} \bigcap \FF\LA X\RA \subset \Id{\A_h}$.

\item[(b)] $\Id{\A_h}  \subset \Id{M_p}$, if $h(\al)\neq0$ for some $\al\in\FF$.

\item[(c)] $\A_h  \not\eqPI M_p$.

\end{enumerate}
\end{theo}
\begin{proof}
Since $\A_h\subset \A_h(\KK)= \A_h \otimes_{\FF}\KK$ as $\FF$-algebras, we can see that 
$$\Id{\A_h \otimes_{\FF}\KK}\; = \; \IdK{\A_h(\KK)} \cap \FF\LA X\RA \; \subset \; \Id{\A_h}.$$%
Part (b) of Theorem~\ref{Teorema Principal} concludes the proof of part (a). Part (b) follows from part (b) of Proposition \ref{Id(Ah) subset Id(Mp)}. 

Consider  $F_{p,q}(x,y)=G_{p,q}(x)\, R_{p,q}(x,y)\, (y^{q}-y)$ of $\FF\LA x,y\RA$, where 
\begin{align*}
G_{p,q}(x)&=(x^{q^{2}}-x)(x^{q^{3}}-x)\cdots(x^{q^{p}}-x),\\
R_{p,q}(x,y)&= \left(1-(y\,({\rm ad}\, x)^{p-1})^{q-1}\right) 
\left(1-(y\,({\rm ad}\, x)^{p-2})^{q-1}\right) \cdots 
\left(1-(y\,{\rm ad}\, x)^{q-1}\right)
\end{align*}
for $y\,{\rm ad}\,x=[y,x]$. 
Genov~\cite{G. Genov} proved that $F_{p,q}(x,y)$ is a polynomial identity for $M_{p}$. 

Since $x\,{\rm ad}\,x = [x,x] = 0$, for $x\in\A_h$ we have $R_{p,q}(x,x)=1$ and  $$F_{p,q}(x,x)=(x^{q}-x)(x^{q^{2}}-x)(x^{q^{3}}-x)\cdots(x^{q^{p}}-x).$$%
By part (b) of Lemma~\ref{lemma_embedding} elements $x,x^2,x^3,\ldots $ are linearly independent in $\A_h$. Therefore, $F_{p,q}(x,x)\neq0$ in $\A_h$; part (c) is proven.
\end{proof}

\begin{conj}\label{conj} 
$\Id{M_p(\KK)} \bigcap \FF\LA X\RA = \Id{\A_h}$.
\end{conj}
\medskip

\section{Counterexample}\label{section_example}

In this section we consider a counterexample to show that without Convention~\ref{conv1} Theorems~\ref{theo0} and~\ref{Teorema Principal} do not hold. Namely, we consider the commutative algebra $\B\simeq\FF^2$ of double numbers from Example~\ref{ex2}, i.e., $\B$ has an $\FF$-basis $\{1,\zeta\}$ with $\zeta^2=0$, and set $h=\zeta$. Note that  Convention~\ref{conv1} does not hold for $h$. Then the statements of Theorems~\ref{theo0} and \ref{Teorema Principal} are not valid for  $\A_{\zeta}(\FF^2)=\A_h(\B)$  (see Proposition~\ref{prop_ex} below).

\begin{remark}\label{remark_ex}
If Convention~\ref{conv1} does not hold for $h$, then Lemmas~\ref{lemma_basis} and \ref{lemma_mult} are still valid for $\A_h(\B)$.
\end{remark}
\bigskip

Part (b) of Lemma~\ref{lemma_mult} together with Remark~\ref{remark_ex} implies that for all $i,j,r,s\geq 0$ we have
\begin{eq}\label{eq1_ex}
x^i \y^j\cdot x^r \y^s =  x^{i+r}\, \y^{s+j} + \zeta\, jr \, x^{i+r-1}\, \y^{s+j-1},
\end{eq}
where we use conventions that $x^{-1}=0$ and $y^{-1}=0$. Then
\begin{eq}\label{eq2_ex}
[x^i \y^j, x^r \y^s] =   \zeta\, (jr-is) \, x^{i+r-1}\, \y^{s+j-1} \;\text{ in }\; \A_{\zeta}(\FF^2).
\end{eq}

The unital finite dimensional Grassmann algebra $\G_k$ of rank $k$ has   an $\FF$-basis 
$$\{1,e_{i_1}\cdots e_{i_m} \tq 1\leq i_1<\cdots < i_m\leq k\}$$ 
and satisfies the defining relations $e_i^2=0$ and $e_i e_j = - e_j e_i$ for all $1\leq i,j\leq k$. The polynomial identities for $\G_k$ were described by  Di Vincenzo~\cite{DiVincenzo_1991} for $p=0$ and by Giambruno, Koshlukov~\cite{GiambrunoKoshlukov_2001} for any infinite field.


\begin{prop}\label{prop_ex} Assume that $\FF$ is an infinite field. 
\begin{enumerate}
\item[(a)] The T-ideal of identities $\Id{\A_{\zeta}(\FF^2)}$ is generated by $$ f_1=[[x_1,x_2],x_3],\quad   f_2=[x_1,x_2]\, [x_3,x_4].$$

\item[(b)] $\A_{\zeta}(\FF^2)\not\eqPI M_t(\C)$ for every $t\geq 2$ and every $\FF$-algebra $\C$ with unity.

\item[(c)] $\A_{\zeta}(\FF^2)\eqPI \G_k$ if and only if  $k\in\{2,3\}$. 
\end{enumerate}
\end{prop}
\begin{proof}  
\noindent{\bf (a)}  By $\FF$-linearity formula~\Ref{eq2_ex} implies that $[a,b]$ belongs to $\zeta\, \A_{\zeta}(\FF^2)$ for all $a,b\in \A_{\zeta}(\FF^2)$. Then $f_1,f_2\in \FF\LA X\RA$ are nontrivial polynomial identities for $\A_{\zeta}(\FF^2)$, since $\zeta^2=0$. Note that 
$$f_3 = [x_1,x_2]\,x_3\, [x_4,x_5] = [[x_1,x_2],x_3] [x_4,x_5] + x_3 [x_1,x_2] [x_4,x_5]\in \Id{\A_{\zeta}(\FF^2)}$$
follows from $f_1,f_2$. Denote by $I$ the T-ideal generated by $f_1,f_2$. 

Assume that $f=\sum_k \al_k w_k$ is a nontrivial identity for $\A_{\zeta}(\FF^2)$, where $\al_k\in\FF$ and $w_k\in\FF\LA x_1,\ldots,x_m\RA$ is a monomial. Since $\FF$ is infinite, we can assume that $f$ is multihomogeneous. i.e., there exists $\un{d}\in\NN^m$ with $\mdeg(w_k)=\un{d}$ for each $k$. We apply equalities  
$$u x_j x_i v = u x_i x_j v - u [x_i,x_j] v,$$
$$u [x_i,x_j] v = [x_i,x_j] u v - [[x_i,x_j], u] v,$$
where monomials $u,v$ can be empty and $i<j$, to monomials $\{w_k\}$ and then repeat this procedure.  Since $f_1,f_3\in \Id{\A_{\zeta}(\FF^2)}$, we finally obtain that there exist $g\in I$, $\al_0, \al_{ij}\in\FF$ such that
$$f = g + \al_0 x_1^{d_1}\cdots x_m^{d_m} + \sum_{1\leq i<j\leq m} \al_{ij} [x_i,x_j] x_1^{d_1}\cdots x_i^{d_i-1} \cdots x_j^{d_j-1} \cdots x_m^{d_m} \;\text{ in }\;\FF\LA X\RA.$$
Since $f(1,\ldots,1) = g(1,\ldots,1) + \al_0$, we obtain that $\al_0=0$. Consider $i<j$ with $d_i,d_j\geq1$.  Making substitutions $x_i\to x$, $x_j\to \y$, $x_l\to 1$ for each $l$ different from $i$ and $j$, we can see that $0=0 + \al_{ij} [x,\y] x^{d_i-1}\y^{d_j-1}$ in $\A_{\zeta}(\FF^2)$. Thus $-\al_{ij} \zeta x^{d_i-1}\y^{d_j-1}=0$  in $\A_{\zeta}(\FF^2)$. Lemma~\ref{lemma_basis} together with Remark~\ref{remark_ex} implies that $\al_{ij}=0$. Therefore, $f=g$ lies in $I$.

\medskip
\noindent{\bf (b)} Since $\FF\subset \C$, every polynomial identity for $M_t(\C)$ lies in $\Id{M_t(\FF)}$. By Amitsur--Levitzki Theorem~\cite{Amitsur_Levitzki} the minimal degree of a polynomial identity for $M_t(\FF)$ is $2t$. In particular, $f_1$ is not an identity for $M_t(\C)$. 

\medskip
\noindent{\bf (c)} Since $\G_k$ is commutative in case $p=2$ or $k=1$, we can assume that $p\neq 2$ and $k\geq2$. Note that 
\begin{eq}\label{eq3_ex}
f_2(e_1,e_2,e_3,e_4)=4e_1 e_2 e_3 e_4 \neq 0 \;\text{ in }\; \G_k \;\text{ for }\; k\geq4.
\end{eq}
Thus we can assume that $k\in\{2,3\}$. The T-ideal $\Id{\G_k}$ is generated by 
\begin{enumerate}
\item[$\bullet$] $f_1,f_2$ in case $p=0$ or $p=k=3$. 

\item[$\bullet$] $f_1,{\rm St}_4$ in case $p>k$, where $k\in\{2,3\}$. 
\end{enumerate}
Since 
$${\rm St}_4(x_1,x_2,x_3,x_4) = [x_1,x_2]\circ [x_3,x_4] - [x_1,x_3]\circ [x_2,x_4] + [x_1,x_4] \circ [x_2,x_3],$$
where $u\circ v$ stands for $uv+vu$, part (a) implies that ${\rm St}_4$ lies in $\Id{\A_{\zeta}(\FF^2)}$. On the other hand, we can see that $f_2$ is a polynomial identity for $\G_k$ when $k\in\{2,3\}$. Part (c) is proven.
\end{proof}


\begin{thebibliography}{99}

\bibitem{AD2}J. Alev, F. Dumas, {\it Invariants du corps de Weyl sous l’action de groupes finis} (French, with English summary), Communications in Algebra {\bf 25} (1997), no. 5, 1655--1672.

\bibitem{Amitsur_Levitzki}A.S. Amitsur,  J. Levitzki, {\it 
Minimal identities for algebras}, Proc. Amer. Math. Soc. {\bf 1} (1950), 449--463.

\bibitem{AVV}M. Awami, M. Van den Bergh, and F. Van Oystaeyen, {\it Note on derivations of graded rings and classification of differential polynomial rings}, Bull. Soc. Math. Belg. S\'er. A {\bf 40} (1988), no. 2, 175--183. Deuxi\`eme Contact Franco-Belge en Alg\`ebre (Faulx-les-Tombes, 1987). 

\bibitem{Bavula_1992}V.V. Bavula, {\it Generalized Weyl algebras and their representations} (Russian), Algebra i Analiz {\bf 4} (1992), no. 1, 75--97. English translation: St. Petersburg Math. J. {\bf 4} (1993), no. 1, 71--92.

\bibitem{Bavula_2020}V.V. Bavula, {\it Classification of simple modules of the Ore extension $K[X][Y; f\frac{d}{dX}]$}, Math. Comput. Sci. {\bf 14} (2020), 317--325. 
\bibitem{Bavula_2021_zero}V.V. Bavula, {\it Isomorphism problems and groups of automorphisms for Ore extensions K[x][y; $\delta$] (zero characteristic)}, arXiv: 2107.09401.

\bibitem{Bavula_2021_prime}V.V. Bavula, {\it Isomorphism problems and groups of automorphisms for Ore extensions K[x][y; $f\frac{d}{dx}$] (prime characteristic)}, arXiv: 2107.09977.

\bibitem{Benkart_Lopes_Ondrus_II}G. Benkart, S.A. Lopes, M. Ondrus, {\it A parametric family of subalgebras of the Weyl algebra II. Irreducible modules},  Recent developments in algebraic and combinatorial aspects of representation theory, 73--98, Contemp. Math., {\bf 602}, Amer. Math. Soc., Providence, RI, 2013.
	
\bibitem{Benkart_Lopes_Ondrus_I}G. Benkart, S.A. Lopes, M. Ondrus, {\it A parametric family of subalgebras of the Weyl algebra I. Structure and automorphisms}, Transactions of the American Mathematical Society {\bf 367} (2015), no.~3, 1993--2021.


\bibitem{Benkart_Lopes_Ondrus_III}G. Benkart, S.A. Lopes, M. Ondrus, {\it Derivations of a parametric family of subalgebras of the Weyl algebra}, Journal of Algebra {\bf 424} (2015), 46–-97.

\bibitem{DiVincenzo_1991}O.M. Di Vincenzo, {\it A note on the identities of the Grassmann algebras}, Unione Matematica Italiana. Bollettino. A. Serie VII, {\bf 5} (1991), no. 3, 307--315.


\bibitem{Askar_2004}A.S. Dzhumadil'daev, {\it $N$-commutators}, Comment. Math. Helv. {\bf 79} (2004), no. 3, 516--553.

\bibitem{Askar_2014}A.S. Dzhumadil'daev, {\it $2p$-commutator on differential operators of order $p$}, Lett. Math. Phys. {\bf 104} (2014), no. 7, 849--869,

\bibitem{Askar_Yeliussizov_2015}A.S. Dzhumadil'daev, D. Yeliussizov, {\it Path decompositions of digraphs and their applications to Weyl algebra}, Adv. in Appl. Math. {\bf 67} (2015), 36--54.

\bibitem{W1_2015}J.A. Freitas, P. Koshlukov, A. Krasilnikov, {\it 
$\ZZ$-graded identities of the Lie algebra $W_1$}, Journal of Algebra {\bf 427} (2015), 226--251.

\bibitem{G. Genov}G. Genov, {\it Basis for identities of a third order matrix algebra over a finite field}, Algebra Log. {\bf 20} (1981), 241--257.

\bibitem{GiambrunoKoshlukov_2001}A. Giambruno, P. Koshlukov, {\it  On the identities of the Grassmann algebras in characterisitc  $p>0$}, Israel Journal of Mathematics {\bf 122} (2001), 305--316.

\bibitem{A.Gia}A. Giambruno, M. Zaicev, {\it Polynomial identities and asymptotic methods}, Math. Surveys Monographs vol. 122, AMS, 2005.

\bibitem{K. R. Goodearl}K.R. Goodearl, R.B. Warfield  Jr., {\it An Introduction to Noncommutative Noetherian Rings}, 2nd edition, Cambridge University Press, 2004.

\bibitem{Lopes_Solotar_2019}S.A. Lopes, A. Solotar, {\it Lie structure on the Hochschild cohomology of a family of subalgebras of the Weyl algebra}, arXiv: 1903.01226.

\bibitem{Mishchenko_1989}S.P. Mishchenko, {\it Solvable subvarieties of a variety generated by a Witt algebra}, Math. USSR Sb. {\bf 64} (1989), no. 2, 415--426.

\bibitem{Razmyslov_book}Yu. Razmyslov, {\it Identities of algebras and their representations}, Transl. Math. Monogr., vol. 138, Amer. Math. Soc., Providence, RI, 1994.


	
	

	

	
	
	
	
	

\end{thebibliography}

\end{document}